\documentclass[12pt,leqno,amscd, amsfonts, amssymb,pstricks,verbatim]{amsart}
\usepackage{bbm}
\usepackage{}
\usepackage{amsfonts}
\usepackage{color}

\usepackage{mathrsfs}
\usepackage{amsmath}
\usepackage{amssymb}

\usepackage[shortlabels]{enumitem}
\usepackage{ifpdf}
\ifpdf
 \usepackage[colorlinks,final,backref=page,hyperindex]{hyperref}
\else
  \usepackage[colorlinks,final,backref=page,hyperindex,hypertex]{hyperref}
\fi
\usepackage{tikz}
\usepackage[active]{srcltx}

\newtheorem{thm}{Theorem}[section]
\newtheorem{lem}[thm]{Lemma}
\newtheorem{cor}[thm]{Corollary}
\newtheorem{prop}[thm]{Proposition}

 \theoremstyle{definition}
\newtheorem{example}[thm]{Example}

\newtheorem{defn}[thm]{Definition}

\newtheorem{rem}[thm]{Remark}

\numberwithin{equation}{section}

\def\ggg{\mathfrak{g}}
\def\hhh{\mathfrak{h}}

\def\ad{{\rm ad}}

\def\Der{{\rm Der}}
\def\Inn{{\rm Inn}}

\def\frakd{{\mathfrak{D}}}

\def\bbf{\mathbb{F}}
\def\bbz{\mathbb{Z}}

\begin{document}

\title[2-local derivations on the Jacobson-Witt algebras in prime characteristic]{2-local derivations on the Jacobson-Witt algebras in prime characteristic}

\author{Yufeng Yao and Kaiming Zhao}

\address{Department of Mathematics, Shanghai Maritime University,
 Shanghai, 201306, China.}\email{yfyao@shmtu.edu.cn}

\address{Department of Mathematics, Wilfrid
Laurier University, Waterloo, ON, Canada N2L 3C5,  and School of Mathematical Sciences, Hebei Normal University, Shijiazhuang 050016, Hebei, China.} \email{kzhao@wlu.ca}

\subjclass[2010]{17B05,17B20,17B40, 17B50, 17B70}

\keywords{the Jacobson-Witt algebra, derivation, 2-local derivation, regular vector, centralizer}

\thanks{This work is supported by the National Natural Science Foundation of China (Grant Nos. 11771279, 11671138 and 11871190) and NSERC (311907-2015)}
\begin{abstract}
This paper initiates the study of 2-local derivations on Lie algebras over fields of prime characteristic. Let $\mathfrak{g}$ be a simple Jacobson-Witt algebra $W_n$ over a field of prime characteristic $p$ with cardinality no less than $p^n$. In this paper, we study properties of  2-local derivations on $\mathfrak{g}$, and  show that every 2-local derivation on $\mathfrak{g}$ is a derivation. 
\end{abstract}

\maketitle
\renewcommand{\theequation}{\arabic{section}.\arabic{equation}}

\section{Introduction}
As it is known, the derivation algebra of an   algebra $A$ plays an important role in the study of the structure of $A$.
In the theory of Lie algebras, a well-known result due to H. Zassenhaus states that all derivations on a finite dimensional Lie algebra with
nondegenerate Killing form are inner (cf. \cite{Jac}). In particular, finite dimensional  semisimple Lie algebras over an algebraically closed field of characteristic zero admit only inner derivations. Hence, they are isomorphic to their derivation algebras.

As a generalization of derivation, $\check{\text{S}}$emrl introduced the notion of 2-local derivation on algebras in \cite{Se}.
The concept of 2-local derivation is actually an important and interesting property for an algebra. The main problem in this subject is to determine
all 2-local derivations, and to see whether they are automatically (global) derivations. All 2-local derivations on several important classes of
Lie algebras have been determined. In \cite{ AKR}, it was shown that each 2-local derivation on a finite dimensional semisimple Lie algebra over an
algebraically closed field of characteristic zero is a derivation and each finite dimensional nilpotent Lie algebra with dimension larger than
two admits a 2-local derivation which is not a derivation. Furthermore, the authors in \cite{WCN} proved that all 2-local derivations on finite
dimensional basic classical Lie superalgebras except $A(n,n)$ over an algebraically closed field of characteristic zero are derivations.
Similar results on 2-local derivations on simple Leibniz algebras were obtained in \cite{AKO}.
All 2-local derivation on  Witt algebras and some of their subalgebras were shown to be derivations in \cite{ZCZ, AKY}.
Similar result was obtained quite recently for the $W$-algebra $W(2,2)$ in \cite{Ta}. In the present paper, we initiate the study of
2-local derivations on finite dimensional Lie algebras over an infinite field of positive characteristic. The algebras we concern   are the
so-called Jacobson-Witt algebras, which are the modular version of some generalized Witt algebras. Let us briefly introduce them below.

Different from the situation of characteristic zero, besides classical simple Lie algebras, there is another variety of simple Lie algebras,
the so-called simple Lie algebras of Cartan type, in the classification of finite dimensional simple Lie algebras over an algebraically
closed field $\mathbb{F}$ of prime characteristic $p>5$ (cf. \cite{PS}). The Lie algebras of Cartan type consist of four
families $W, S, H, K$ (cf. \cite{SF, St}). The algebras we focus on in the present paper are the first series.
The Jacobson-Witt algebra $W_n$ is the derivation algebra of the truncated polynomial algebra
$\mathfrak{A}_n=\bbf[x_1,\cdots, x_n]/(x_1^p,\cdots, x_n^p)$, where $(x_1^p,\cdots, x_n^p)$ is the ideal of $\bbf[x_1,\cdots, x_n]$ generated
by $x_i^p$, $1\leq i\leq n$. Then $W_n$ is a simple Lie algebra unless $n=1$ and $p=2$. Over the past decades, the representation theory of the Jacobson-Witt algebras was extensively studied (see \cite{Sh, HZ, SY}). The derivation algebra of $W_n$ was completely determined for $p>2$ (see \cite{SF, St}). This paper is devoted to studying 2-local derivations on $W_n$. Under a mild restriction on the size of the base field, we determine all 2-local derivations on the simple Jacobson-Witt algebras, and show that each 2-local derivation is a (global) derivation.

This paper is organized as follows. In section 2, we recall the basic notations, definitions, structure and some important properties of the Jacobson-Witt algebras. Section 3 is devoted to studying 2-local derivations on the simple
Jacobson-Witt algebras $W_n$ over a field of prime characteristic $p$ with cardinality no less than $p^n$. We present some properties of 2-local derivations, and show that every 2-local derivation on any simple Jacobson-Witt algebra is a derivation. Moreover, we give an example of the Jacobson-Witt algebra of rank 1, i.e. the so-called Witt algebra over a field of characteristic $p=2$, in which there exists a 2-local derivation that is not a derivation.

Similar to the study on structure of simple Lie algebras of positive characteristic, the study  on 2-local derivations of   Lie algebras of positive characteristic is very different and more difficult than the case of   characteristic $0$. We have to establish new and different methods (Lemmas 3.3 and 3.14) to achieve our goal.

\section{Notations and preliminaries}
In this paper, we always assume that  $\bbf$ is a field  of positive characteristic $p$, and let $\bbf_p$ denote the prime subfield of $\bbf$, and $\bbf^*=\bbf\setminus\{0\}$. Throughout this paper, all algebras and vector spaces  are over $\bbf$ and finite dimensional. We denote by $\mathbb{Z},  \mathbb{N},  \mathbb{Z}_+$ the set of all integers, nonnegative integers and positive integers respectively. For a set $S$, we use $|S|$ 
to denote the cardinality of $S$.

\subsection{Derivations and 2-local derivations on a Lie algebra}
A {\bf derivation} on a Lie algebra $\ggg$  is a linear transformation $D:\ggg\longrightarrow\ggg$ such that the following Leibniz law holds:
$$D([x, y])=[D(x), y]+[x, D(y)],\,\,\forall\,x,y\in\ggg.$$
The set of all derivations of $\ggg$ is denoted by $\Der(\ggg)$, which is a Lie algebra under the usual commutant operation. For each $x\in\ggg$, let
$$\ad x:\ggg\longrightarrow\ggg,\,\,\,\ad x (y)=[x, y], \,\,\forall\,y\in\ggg.$$ Then $\ad x$ is a derivation on $\ggg$ for any $x\in\ggg$, which is  called an inner derivation. The set of all inner derivations of $\ggg$ is denoted by $\Inn(\ggg)$,
which is an ideal of $\Der(\ggg)$.

A map $\Delta: \ggg\longrightarrow\ggg$ (not necessarily linear) is called  a {\bf 2-local derivation} if for any $x,y\in\ggg$,
there exists a derivation $D_{xy}\in\Der(\ggg)$ (depending on $x, y$) such that $\Delta(x)=D_{xy}(x)$ and $\Delta(y)=D_{xy}(y)$.
In particular, for any $x\in\ggg$ and $k\in\bbf$, there exists $D_{(kx)x}\in\Der(\ggg)$ such that
$$\Delta(kx)=D_{(kx)x}(kx)=kD_{(kx)x}(x)=k\Delta(x).$$
In particular,
\begin{equation}\label{a property}
\Delta(0)=0.
\end{equation}
Hence, a 2-local derivation $\Delta$ on $\ggg$ is a derivation if and only if $\Delta$ is additive and satisfies the Leibniz law, i.e.,
$$\Delta(x+y)=\Delta(x)+\Delta(y), \,\Delta([x,y])=[\Delta(x), y]+[x, \Delta(y)],\,\,\forall\,x,y\in\ggg.$$

\subsection{The Jacobson-Witt algebras} In this subsection, we recall the basic definitions and properties of the
Jacobson-Witt algebras which we concern in this paper. We use the terminology and notations in
\cite{SF, St}. For $n\in\mathbb{Z}_+$, set $$A_n=\{\alpha=(\alpha_1,\cdots, \alpha_n)\in\mathbb{N}^n :  0\leq \alpha_i\leq p-1, 1\leq i\leq n\},$$
$$\tau=(p-1, \cdots, p-1), \,\,\epsilon_i=(\delta_{i1},\cdots, \delta_{in})\,\text{ for}\,\, 1\leq i\leq n,$$
where
\begin{align*}
\delta_{ij}=\begin{cases} 1, &\mbox{ if } i=j;\cr
0, &\mbox{ otherwise.}
\end{cases}
\end{align*}
Let ${\mathfrak{A}}_n=\bbf[x_1,\cdots, x_n]/(x_1^p,\cdots, x_n^p)$ be the truncated polynomial algebra of $n$ variables $x_1,\cdots, x_n$, where $(x_1^p,\cdots, x_n^p)$ denotes the ideal of $\bbf[x_1,\cdots, x_n]$ generated by $x_i^p, 1\leq i\leq n$. For $\alpha=(\alpha_1,\cdots,\alpha_n)\in A_n$, set $|\alpha|=\sum_{i=1}^n\alpha_i, $ and use $x^{\alpha}:=x_1^{\alpha_1}\cdots x_n^{\alpha_n}$ to denote its canonical image in $\mathfrak{A}_n$ for brevity. Then $\mathfrak{A}_n$ has a basis $\{x^{\alpha} :  \alpha\in A_n\}$ with the multiplication subject to $x^{\alpha}x^{\beta}=x^{\alpha+\beta}$
with the convention that $x^{\alpha+\beta}=0$ if $\alpha+\beta\notin A_n$. Moreover, $\mathfrak{A}_n$ has a natural $\mathbb{Z}$-grading
$$\mathfrak{A}_n=\bigoplus\limits_{i=0}^{n(p-1)}(\mathfrak{A}_n)_{[i]},$$
where $(\mathfrak{A}_n)_{[i]}=\text{span}_{\mathbb{F}}\{x^{\alpha} :  |\alpha|=i\}$. For $1\leq i\leq n$, let $D_i$ be the linear transformation on $\mathfrak{A}_n$ with $D_i(x^{\alpha})=\alpha_i x^{\alpha-\epsilon_i}$ for any $\alpha\in A_n$. Then it is easy to see that $D_i\in\Der(\mathfrak{A}_n)$ for $1\leq i\leq n$. The {\bf Jacobson-Witt algebra} $W_n$ is defined as the derivation algebra of $\mathfrak{A}_n$, i.e., $W_n=\Der(\mathfrak{A}_n)$. Then by \cite[\S\,4.2]{SF}, $W_n$ is a free $\mathfrak{A}_n$-module of rank $n$ with a basis $\{D_1,\cdots, D_n\}$. The Lie bracket in $W_n$ is given by
$$[fD_i, gD_j]=f(D_i(g))D_j-g(D_j(f))D_i,\,\,f,g\in\mathfrak{A}_n, 1\leq i,j\leq n.$$
Moreover, $W_n$ is a simple Lie algebra unless $n=1$ and $p=2$. The natural $\mathbb{Z}$-grading on $\mathfrak{A}_n$ induces the corresponding $\mathbb{Z}$-grading structure on $W_n$,
$$W_n=\bigoplus\limits_{i=-1}^{n(p-1)-1}(W_n)_{[i]},$$
where
$$(W_n)_{[i]}=\text{span}_{\mathbb{F}}\{x^{\alpha}D_j :  |\alpha|=i+1, 1\leq j\leq n\}.$$
Furthermore, $W_n$ has a canonical torus $T=\sum_{i=1}^n\bbf x_iD_i\in (W_n)_{[0]}$, and it has the following root space decomposition with respect to the torus $T$:
\begin{equation}\label{decomp for W}
W_n=T\oplus\big(\bigoplus\limits_{\alpha\in\Lambda}(W_n)_{\alpha}\big),
\end{equation}
where
$$(W_n)_{\alpha}=\text{span}_{\bbf}\{x^{\alpha+\epsilon_j}D_j :  1\leq j\leq n\},$$
and $\Lambda=\{a_1\epsilon_1+\cdots+a_n\epsilon_{n}-\epsilon_i :  0\leq a_1,\cdots, a_n\leq p-1, 1\leq i\leq n\}\setminus\{0\}$ is the set of all roots.

We need the following result on the derivation algebras of the Jacobson-Witt algebras for later use.

\begin{lem}\label{derivation lem}
We have $\Der(W_n)=\Inn(W_n)$ for any $n\in\mathbb{Z}_+$.
\end{lem}

\begin{proof}
When the base field $\bbf$ is of characteristic $p>2$, the assertion follows from \cite[Theorems 8.5, Chapter 4]{SF} (also see \cite[Theorem 7.1.2]{St}).

In the following we assume that the base field $\bbf$ is of characteristic $p=2$. We will refine the proof of \cite[Theorem 8.5, Chapter 4]{SF} to the case of characteristic $p=2$. In this situation, we divide the discussion into the following two cases.

Case 1: $n=1$

In this case $W_1=\text{span}_{\bbf}\{D_1, x_1D_1\}$, and $[W_1, W_1]=\bbf D_1$. For any $D\in\Der(W_1)$, we can assume $D(D_1)=aD_1, D(x_1D_1)=bx_1D_1+cD_1$ for $a, b,c\in\bbf$. Then
$$aD_1=D(D_1)=D([D_1, x_1D_1])=[D(D_1), x_1D_1]+[D_1, D(x_1D_1)]=(a+b)D_1.$$
This implies that $b=0$, and $D=\ad (cD_1-ax_1D_1)$. Consequently, $\Der(W_1)=\Inn(W_1)$.

Case 2: $n\geq 2$.

In this case, let $\ggg=W_n$, then $\ggg=\sum_{i=-1}^{n-1}\ggg_{[i]}$. To show the assertion, from the same arguments in the proof of \cite[Theorem 8.5, Chapter 4]{SF}, it suffices to prove that any derivation $\varphi$ on $\ggg$ of homogeneous degree $t\leq -2$ is trivial. Indeed, $\varphi(\ggg_{[-1]}\oplus\ggg_{[0]})=0$ by definition. We claim that $\varphi(\ggg_{[1]})=0$. Otherwise, without loss of generality  we may assume that $$\varphi(x_1x_2D_i)=\sum_{j=1}^na_jD_j\neq 0,\,\,\text{ for some }i\in\{1,\cdots, n\}, a_j\in\bbf\text{ with }1\leq j\leq n.$$ We need to consider  the following two subcases which may occur.

Subcase 2.1: $i=1$ or $2$.

Without loss of generality, we may assume that $i=1$. In this subcase, we have
$$\aligned \sum_{j=1}^na_jD_j&=\varphi(x_1x_2D_1)=\varphi([x_2D_2, x_1x_2D_1])\\
&=[x_2D_2, \varphi(x_1x_2D_1)]=[x_2D_2, \sum_{j=1}^na_jD_j]=-a_2D_2.\endaligned$$
It follows that $a_j=0$ for any $j\neq 2$, i.e., $\varphi(x_1x_2D_1)=a_2D_2$. Furthermore,
$$0=\varphi([x_2D_1, x_1x_2D_1])=[x_2D_1, \varphi(x_1x_2D_1)]=[x_2D_1, a_2D_2]=-a_2D_1.$$
It follows that $a_2=0$, i.e., $\varphi(x_1x_2D_1)=0$, a contradiction.

Subcase 2.2:  $i>2$.

In this subcase, we have

$$\aligned-\sum_{j=1}^na_jD_j&=-\varphi(x_1x_2D_i)=\varphi([x_iD_i, x_1x_2D_i])\\
&=[x_iD_i, \varphi(x_1x_2D_i)]=[x_iD_i, \sum_{j=1}^na_jD_j]=-a_iD_i.\endaligned$$
It follows that $\varphi(x_1x_2D_i)=a_iD_i$. Hence,
$$a_iD_i=\varphi(x_1x_2D_i)=\varphi([x_1D_1, x_1x_2D_i])=[x_1D_1, \varphi(x_1x_2D_i)]=[x_1D_1, a_iD_i]=0.$$
This implies that $\varphi(x_1x_2D_i)=0$, a contradiction.

In conclusion, $\varphi(\ggg_{[1]})=0$. Since $\ggg$ is generated by $\ggg_{[-1]}$ and $\ggg_{[1]}$ by \cite[Lemma 7.1, Chapter 4]{SF}, it follows that $\varphi=0$, so that $\Der(W_n)=\Inn(W_n)$.
\end{proof}

By Lemma \ref{derivation lem}, we can reformulate the definition of 2-local derivation on the Jacobson-Witt algebra as follows.
Let $\ggg=W_n$ be the Jacobson-Witt algebra. A map $\Delta$ on $\ggg$ is a 2-local derivation if for any two elements $x, y\in\ggg$, there exists an element $a_{x,y}\in\ggg$ such that $\Delta(x)=[a_{xy}, x]$ and $\Delta(y)=[a_{xy}, y]$.

\section{2-local derivations on the Jacobson-Witt algebras}
Throughout  this section, we   assume that $\ggg=W_n$ is the simple Jacobson-Witt algebra over a field  $\bbf$ of prime characteristic $p$ with cardinality no less than $p^n$, that is, we exclude the case that $\ggg=W_1$ and $p=2$. We shall determine all 2-local derivations on $\ggg$.

In general, for an element $x$ in a Lie algebra $\mathfrak{L}$, the centralizer of $x$ in $\mathfrak{L}$ is defined as $\mathfrak{z}_{\mathfrak{L}}(x)=\{y\in\mathfrak{L} :  [x,y]=0\}.$ Then $\mathfrak{z}_{\mathfrak{L}}(x)$ is a subalgebra of $\mathfrak{L}$ containing $x$ itself. For $\lambda=(\lambda_1, \cdots, \lambda_n),
\mu=(\mu_1,\cdots, \mu_n)\in\bbf^n$, let $(\lambda, \mu)=\sum_{i=1}^n\lambda_i\mu_i$.

\begin{defn}
A vector $\lambda=(\lambda_1, \cdots, \lambda_n)\in\bbf^n$ is called {\bf regular} if $\lambda_1,\cdots, \lambda_n$ are $\bbf_p$ linearly independent, that is, for $\mu\in\bbf_p^n$, $(\lambda,\mu)=0$ if and only if $\mu=(0,\cdots, 0)$.
\end{defn}

\begin{rem}
The assumption on the cardinality of the base field $\bbf$ assures the existence of regular vectors in $\bbf^n$.
\end{rem}


For $\lambda=(\lambda_1,\cdots, \lambda_n)\in\bbf^n$ and $0\le k\le p-1$,  let
$$\frakd^{(k)}_{\lambda}=\sum_{i=1}^n\lambda_ix_i^kD_i.$$ 
For $1\leq i\leq n$, let
\begin{equation}\label{Di'}
\mathscr{D}_i=D_i+\sum\limits_{j=i}^{n-1}(\prod\limits_{k=i}^jx_k^{p-1})D_{j+1},
\end{equation}

We need the following lemma for later use.

\begin{lem}\label{lem for D_i}
Keep notations as above, then $\mathscr{D}_i=(-1)^{i-1}\mathscr{D}_1^{p^{i-1}}$ for $1\leq i\leq n$.
\end{lem}

\begin{proof}
When the base field $\bbf$ is of characteristic $p>2$, the assertion follows from \cite[Lemma 3(i)]{Pr}.

In the following, we assume that $\bbf$ has characteristic $p=2$. Since $\mathscr{D}_1\in\Der(\mathfrak{A}_n)=\ggg$, it follows that $\mathscr{D}_1^{2^{i-1}}\in\Der(\mathfrak{A}_n)=\ggg$ for $1\leq i\leq n$. It is clear that $\mathscr{D}_n^2={D}_n^2=0$.
To prove the assertion in the lemma, we only need to show that $\mathscr{D}_i^2=\mathscr{D}_{i+1}$ for $i=1,2,\cdots, n-1$.

Direct computations imply that
  $$\mathscr{D}_i(x_j)=\left\{\begin{matrix}
\delta_{i,j},&\text{  if }j\le i, \\
\prod_{k=i}^{j-1}x_k,& \quad\quad\text{ if }  i< j\leq n.
\end{matrix}\right.$$
Furthermore, $$\mathscr{D}_i^2(x_j)=\left\{\begin{matrix}
\delta_{i+1,j},&\text{  if }j\le i+1, \\
\prod_{k=i+1}^{j-1}x_k,& \quad\quad\text{ if }  i+1< j\leq n.
\end{matrix}\right.$$
Consequently, $\mathscr{D}_i^2=\mathscr{D}_{i+1}$ for $i=1,2,\cdots, n-1$. We complete the proof.
\end{proof}

The following result on the structure of centralizers of some special elements in $\ggg$ is crucial to determine 2-local derivations on $\ggg$.

\begin{lem}\label{lem on centralizer}
Let $\lambda=(\lambda_1,\cdots, \lambda_n)\in\bbf^n$ be regular. 
Then the following statements hold.
\begin{itemize}
\item[(1)] $\mathfrak{z}_{\ggg}(\frakd^{(1)}_{\lambda})=T$.\vspace{1mm}
\item[(2)] If $p>2$, then $\mathfrak{z}_{\ggg}(\sum_{i=1}^n x_i^2D_i)\cap T=0$.\vspace{1mm}
\item[(3)] $\mathfrak{z}_{\ggg}(\mathscr{D}_1)=\sum\limits_{i=1}^{n}\bbf \mathscr{D}_i$.
\end{itemize}
\end{lem}

\begin{proof}
(1) Take any $D\in \mathfrak{z}_{\ggg}(\frakd^{(1)}_{\lambda})$. Thanks to (\ref{decomp for W}), we can write $D=D_0+\sum\limits_{\alpha\in\Lambda}D_{\alpha}$
with $D_0\in T$ and $D_{\alpha}\in\ggg_{\alpha}$ for any $\alpha\in\Lambda$. Then
$$0=[\frakd^{(1)}_{\lambda}, D]=\Big[\frakd^{(1)}_{\lambda}, D_0+\sum\limits_{\alpha\in\Lambda}D_{\alpha}\Big]=\sum\limits_{\alpha\in\Lambda}(\lambda,\alpha) D_{\alpha}.$$
Since $\lambda$ is regular, $(\lambda,\alpha)\neq 0$ for any $\alpha\in\Lambda$. It follows that $D_{\alpha}=0$ for any $\alpha\in\Lambda$. This implies that
$\mathfrak{z}_{\ggg}(\frakd^{(1)}_{\lambda})\subseteq T$. On the other hand, it is obvious that $T\subseteq \mathfrak{z}_{\ggg}(\frakd^{(1)}_{\lambda})$. Hence,
$\mathfrak{z}_{\ggg}(\frakd^{(1)}_{\lambda})=T$.


(2) Take any $E=\sum\limits_{j=1}^nk_jx_jD_j\in\mathfrak{z}_{\ggg}(\sum_{i=1}^n x_i^2D_i)\cap T$ with $k_j\in\bbf$ for $1\leq j\leq n$, then
$$0=\Big[\sum\limits_{j=1}^nk_jx_jD_j,\sum_{i=1}^n x_i^2D_i\Big]=\sum\limits_{i=1}^n k_i x_i^2D_i.$$
It follows that $k_i=0$ for any $1\leq i\leq n$. Consequently, $\mathfrak{z}_{\ggg}(\sum_{i=1}^n x_i^2D_i)\cap T=0$.

(3) When $p>2$, this follows from     \cite[Lemma 7(ii)]{Pr}. Now we prove this for $p=2$. It is obvious that $\sum_{i=1}^{n}\bbf \mathscr{D}_i\subset\mathfrak{z}_{\ggg}(\mathscr{D}_1)$ by Lemma \ref{lem for D_i}. On the other hand, for any $X\in\mathfrak{z}_{\ggg}(\mathscr{D}_1)$, we write $X=\sum_{i\ge r}X_{i}\in\ggg$, where $X_{i}\in\ggg_{[i]}$ for $i\geq r$.
Then we may assume that $r\ge 0$. If $X\ne0$, we may assume that $X_r\ne0$. Since  $[X, \mathscr{D}_i]=0$ for any $1\leq i\leq n$ by Lemma \ref{lem for D_i}, we see that
$[X_r,  {D}_i]=0$ for any $1\leq i\leq n$, which is impossible. Hence, the assertion follows.
\end{proof}

\begin{lem}\label{delta lem}
Let $\lambda=(\lambda_1,\cdots,\lambda_n)\in\bbf^n$ be regular, and $\Delta$ be a 2-local derivation on $\ggg$ such that $\Delta(\frakd^{(1)}_{\lambda})=0$. Then for any nonzero element $X= \sum_{\alpha\in S} \sum_{i\in\Gamma_{\alpha}}c_{\alpha, i}x^{\alpha}D_i\in W_n$, where $S\subseteq A_n$ and $\Gamma_{\alpha}\subset\{1,2,\cdots, n\}$ and $c_{\alpha, i}\in\bbf^*$ for $\alpha\in S, i\in\Gamma_{\alpha}$, we have $$\Delta(X) \in \sum_{\alpha\in S} \sum_{i\in\Gamma_{\alpha}}\bbf x^{\alpha}D_i.$$
\end{lem}

\begin{proof}
 For $\frakd^{(1)}_{\lambda}$ and $X$, there exists an element $a\in\ggg$ such that $\Delta(\frakd^{(1)}_{\lambda})=[a, \frakd^{(1)}_{\lambda}]$ and
$\Delta(X)=[a, X]$. Since  $\Delta(\frakd^{(1)}_{\lambda})=0$, it follows that $a\in T$ by Lemma \ref{lem on centralizer}. Thus
$$\Delta(X)=\bigg[a, \sum_{\alpha\in S} \sum_{i\in\Gamma_{\alpha}}c_{\alpha, i}x^{\alpha}D_i\bigg]\in\sum_{\alpha\in S} \sum_{i\in\Gamma_{\alpha}}\bbf x^{\alpha}D_i.$$
\end{proof}


\begin{prop}\label{Try1}
Suppose the base field $\bbf$ is of characteristic $p>2$. Let $\nu=(1,1,\cdots, 1)\in\bbf^n$, and $X=\sum_{i\ge -1}X_{i}\in\ggg$, where $X_{i}\in\ggg_{[i]}$ for $i\geq -1$. Then the following statements hold.
\begin{itemize}
\item[(1)] If $X\in \mathfrak{z}_{\ggg}(\sum\limits_{i=1}^nx_i^2D_i)$,  then  $X_{-1}=X_{0}=0$.

\item[(2)]  If $X\in \mathfrak{z}_{\ggg}(\frakd_\nu^{(\frac{p+1}{2})})$, then $X_{k}=0$ for all $k< \frac{p-1}2$.
\end{itemize}
\end{prop}

\begin{proof}
(1) Since  $\sum\limits_{i=1}^nx_i^2D_i\in\ggg_{[1]}$ and
\begin{equation}\label{eq1}
0=\bigg[X, \sum\limits_{i=1}^nx_i^2D_i\bigg]=\sum\limits_{k=-1}^{n(p-1)-1}\bigg[X_{k}, \sum\limits_{i=1}^nx_i^2D_i\bigg],
\end{equation}
it follows that
\begin{equation}\label{eq2}
\bigg[X_{k}, \sum\limits_{i=1}^nx_i^2D_i\bigg]=0, \,\,-1\leq k\leq n(p-1)-1.
\end{equation}

Write  $X_{-1}=\sum\limits_{s=1}^n a_sD_s$ for $a_s\in\bbf, 1\leq s\leq n$. Then
$$0=\bigg[X_{-1}, \sum\limits_{i=1}^nx_i^2D_i\bigg]=\sum\limits_{i=1}^n 2a_ix_iD_i,$$
which implies that $a_s=0$, $-1\leq s\leq n$. Hence, $X_{-1}=0$.

Write
$$X_{0}=\sum\limits_{1\leq i,j\leq n}a_{ij}x_iD_j, a_{ij}\in\bbf, 1\leq i, j\leq n.$$
If there exists some $s\neq t$ such that $a_{st}\neq 0$, then $2a_{st}x_sx_tD_t$ appears as a summand in
$[X_{0}, \sum_{i=1}^nx_i^2D_i]$, and can not be cancelled by other summands. This contradicts with (\ref{eq2}) in the case $k=0$. Hence, $X_{0}\in T$. Then it follows from (\ref{eq2}) in the case $k=0$ and Lemma \ref{lem on centralizer}(2) that $X_{0}=0$.

(2) Since  $\frakd_\nu^{(\frac{p+1}{2})}\in\ggg_{[\frac{p-1}{2}]}$ and
\begin{equation}
0=\Big[X, \frakd_\nu^{(\frac{p+1}{2})}\Big]=\sum\limits_{k=-1}^{n(p-1)-1}\bigg[X_{k}, \frakd_\nu^{(\frac{p+1}{2})}\bigg],
\end{equation}
it follows that
\begin{equation}
\Big[X_{k}, \frakd_\nu^{(\frac{p+1}{2})}\Big]=0, \,\,-1\leq k\leq n(p-1)-1.
\end{equation}

Assume $X_{k}=\sum\limits_{i=1}^n f_i^{(k)}D_i$, where $f_i^{(k)}\in (\mathfrak{A}_n)_{[k+1]}, 1\leq i\leq n$. Then
$$0=\Big[X_{k}, \frakd_\nu^{(\frac{p+1}{2})}\Big]=\sum\limits_{i=1}^n\Big({(\frac{p+1}{2})}x_i^{\frac{p-1}{2}}f_i^{(k)}-
\big(\sum\limits_{j=1}^nx_j^{\frac{p+1}{2}}D_j(f_i^{(k)})\big)\Big)D_i.$$
Hence,
$${(\frac{p+1}{2})}x_i^{\frac{p-1}{2}}f_i^{(k)}-\big(\sum\limits_{j=1}^nx_j^{\frac{p+1}{2}}D_j(f_i^{(k)})\big)=0,\,\forall\,1\leq i\leq n.$$
It follows that
$$f_i^{(k)}=0\,\text{ for\, any}\,-1\leq k<\frac{p-1}{2}, 1\leq i\leq n.$$
That is
$$X_{k}=0\,\text{ for}\,-1\leq k<\frac{p-1}{2}.$$
\end{proof}

As a direct consequence of Proposition \ref{Try1}, we have
\begin{cor}\label{Try2}
Suppose the base field $\bbf$ is of characteristic $p>2$ and $\Delta$ is a 2-local derivation on $\ggg$ such that $\Delta(\frakd^{(1)}_{\lambda})=\Delta(\sum_{i=1}^nx_i^2D_i)=0$ for some regular
vector $\lambda\in\bbf^n$. Then
\begin{itemize}
\item[(1)] $\Delta(\ggg_{[k]})=0$ for any $k$.
\item[(2)] $\Delta(\sum_{k\ge r}\ggg_{[k]})\subseteq \sum_{k\ge r+\frac{p-1}2}\ggg_{[k]}$ for any $r$.
\item[(3)] $\Delta(\mathscr{D}_i)=0$ for $1\leq i\leq n$.
\end{itemize}
\end{cor}

\begin{proof}
(1) Let $X\in\ggg_{[k]}$. 
For $\sum_{i=1}^nx_i^2D_i$ and $X$, there exists $Y=\sum_{i\geq -1} Y_{i}\in\ggg$ with $Y_{i}\in\ggg_{[i]}$ for $i\geq -1$ such that $0=\Delta(\sum_{i=1}^nx_i^2D_i)=[Y, \sum_{i=1}^nx_i^2D_i]$ and $\Delta(X)=[Y, X]$. It follows from Proposition \ref{Try1} that $Y_{0}=0$. Since $X, \Delta(X)\in\ggg_{[k]}$, we further obtain from Lemma \ref{delta lem} that
$$\Delta(X)=[Y, X]=[Y_{0}, X]=0.$$

(2) Since $\frakd_\nu^{(\frac{p+1}{2})}\in\ggg_{[\frac{p-1}{2}]}$ for $\nu=(1,1,\cdots, 1)$, it follows from the statement (1) that $\Delta\big(\frakd_\nu^{(\frac{p+1}{2})}\big)=0$. Then for any $X\in\sum_{k\ge r}\ggg_{[k]}$, there exists $Y=\sum_{i\geq -1} Y_{i}\in\ggg$ with $Y_{i}\in\ggg_{[i]}$ for $i\geq -1$ such that
$$0=\Delta(\frakd_\nu^{(\frac{p+1}{2})})=[Y, \frakd_\nu^{(\frac{p+1}{2})}],$$ and $\Delta(X)=[Y, X]$. Note that $Y_{k}=0$ for $k<\frac{p-1}{2}$ by Proposition \ref{Try1}, we further have
$$\Delta(X)=[Y, X]=\sum_{i\geq\frac{p-1}{2}}[Y_i, X]\in\sum_{k\ge r+\frac{p-1}{2}}\ggg_{[k]}.$$

(3) By Lemma \ref{delta lem} and the statement (2), we see that
\begin{equation}\label{Eq1 for D}
\Delta(\mathscr{D}_i)=
\sum\limits_{j=i}^{n-1}l_i^{(j+1)}(\prod\limits_{k=i}^jx_k^{p-1})D_{j+1},
\end{equation}
where $l_i^{(j)}\in\bbf$ for $i< j\leq n$.

On the other hand, for $\sum_{i=1}^nx_i^2D_i$ and $\mathscr{D}_i$, there exists $b=\sum_{i\geq -1}b_{i}\in\ggg$ with $b_{i}\in\ggg_{[i]}$ for $i\geq -1$  such that $0=\Delta\big(\sum_{i=1}^nx_i^2D_i\big)=\big[b, \sum_{i=1}^nx_i^2D_i\big]$ and $\Delta(\mathscr{D}_i)=[b, \mathscr{D}_i]$. Then
$b_{-1}=b_{0}=0$ by Proposition \ref{Try1}. Hence
\begin{equation}\label{Eq2 for D}
\Delta(\mathscr{D}_i)=[b, \mathscr{D}_i]=\sum_{k\geq 1}[b_{k}, D_i+\sum\limits_{j=i}^{n-1}(\prod\limits_{k=i}^jx_k^{p-1})D_{j+1}].
\end{equation}
By comparing the right hand sides of (\ref{Eq1 for D}) and (\ref{Eq2 for D}), we have
$$ l_i^{(j+1)}(\prod\limits_{k=i}^jx_k^{p-1})D_{j+1}=[b_{(p-1)(j-i+1)}, D_i]+
\sum\limits_{s=i}^{j}[b_{(p-1)(j-s)},(\prod\limits_{k=i}^sx_k^{p-1})D_{s+1}], \, i< j\leq n-1.$$
The right-hand-side of this equation does not produce any term of the form
$(\prod_{k=i}^jx_k^{p-1})D_{j+1}$ since $b_0=0$.
This implies that $l_i^{(j)}=0$ for $i< j\leq n$, i.e., $\Delta(\mathscr{D}_i)=0$, as desired.
\end{proof}

We define the support of $X=\sum_{\alpha\in A_n}\sum_{1\le i\le n}a_{\alpha, i}x^\alpha D_i\in\ggg$, where $a_{\alpha, i}\in F$,  as
$${\rm supp}(X):=\{(\alpha,i):a_{\alpha, i}\ne0\}.$$
In this section, from now on, we take a regular vector
 $\lambda=(\lambda_1,\cdots, \lambda_n)\in\bbf^n$, and
  let $\Delta$ be a 2-local derivation on $\ggg$ such that $\Delta(\frakd^{(1)}_{\lambda})=\Delta(\sum\limits_{i=1}^nx_i^2D_i)=0$ if $p>2$, and
  $\Delta(\frakd^{(1)}_{\lambda})=\Delta(\mathscr{D}_1)=0$ if $p=2$.

 We want to show that $\Delta=0$. To the contrary, assume that there is
  $X=\sum_{i= r}^{r+s}X_{i}\in\ggg$ with $X_i\in\ggg_{[i]}$ for $r\leq i\leq r+s$,
$ X_{r}\ne0$, $ X_{r+s}\ne0$ such that $\Delta(X)\ne0$. We need to deduce a contradiction. Thanks to Lemma \ref{delta lem}, we can write
\begin{equation}\label{delta x}
X':=\Delta(X)=\sum_{i= r}^{r+s}X'_{i},
\end{equation}
where $X'_i\in\ggg_{[i]}$ for $r\leq i\leq r+s$. We may choose such an  $X$ so that $r+s$ is maximal and then $s$ is minimal.
We may further assume that
$|{\rm supp}(X_{r+s})|$ and $|{\rm supp}(X_{r+s-1})|$ are  minimal in the sense that, for  any $Y=\sum_{i=r}^{r+s}Y_{i}\in\ggg$, where  $Y_i\in\ggg_{[i]}$ for any $i\geq -1$, if $|{\rm supp}(X_{i})|=|{\rm supp}(Y_{i})|$ for $i\in\{r+s-1, r+s\}\setminus\{k\}$
and $|{\rm supp}(X_{k})|>|{\rm supp}(Y_{k})|$, then  $\Delta(Y)=0$.

%

%
%
%
%
%
%

The following observation is elementary.

\begin{lem}\label{X'} We have
$X'_i=0$ if $i< r+s-1$, that is, $X'= X'_{r+s}+X'_{r+s-1}.$
\end{lem}

\begin{proof}  From the minimal conditions on $X$, we know that $\Delta(X-X_{r+s})=0$. There is an element $a\in\ggg$ such that
$$0=\Delta(X-X_{r+s})=[a,  X-X_{r+s}],\,\,\,
X'=\Delta(X)=[a, X]=[a, X_{r+s}].$$
The statement follows from (\ref{delta x}) and $a=a_{-1}+a_0+\cdots+a_m$ with $a_i\in\ggg_{[i]}$ for $-1\leq i\leq m$.
\end{proof}

\begin{lem}\label{X_{r+s}} If $(\alpha, i)\in {\rm supp}(X_{r+s})$, then
$$ X'_{r+s}\in{\rm span}_{\bbf} \{x^\alpha D_i, x^{\alpha+\epsilon_k-\epsilon_j} D_i,
x^\alpha D_k-\alpha_k x^{\alpha+\epsilon_i-\epsilon_k} D_i:k\ne i, 1\le j\le n\},$$
and
$$X'_{r+s-1}\in{\rm span}_{\bbf}\{x^{\alpha-\epsilon_j}D_i: 1\leq j\leq n\}.$$
\end{lem}

\begin{proof} From the minimal conditions on $X$, we know that $\Delta(X-a_{\alpha, i}x^\alpha D_i)=0$. There is an element $a\in\ggg$ such that
$$0=\Delta(X-a_{\alpha, i}x^\alpha D_i)=[a,  X-a_{\alpha, i}x^\alpha D_i],\,\,\,
X'=\Delta(X)=[a, X]=[a, a_{\alpha, i}x^\alpha D_i].$$
The statement follows from $X'_{r+s}\in [\ggg_{[0]}, a_{\alpha, i}x^\alpha D_i]$ and  $X'_{r+s-1}\in [\ggg_{[-1]}, a_{\alpha, i}x^\alpha D_i]$.
\end{proof}

%

\begin{lem}\label{X_{r+s}-3} If $(\alpha, 1), (\alpha+\epsilon_1-\epsilon_2, 1)\in {\rm supp}(X_{r+s})$, then
$$ {\rm supp}(X'_{r+s})\subset   \{(\alpha+\epsilon_k-\epsilon_2, 1):  2\le k\le n\}.$$
\end{lem}

\begin{proof} The assertion follows directly from Lemma \ref{X_{r+s}}.
\end{proof}

%
%
%
%
%

As a consequence of Lemma \ref{X_{r+s}} and Lemma \ref{X_{r+s}-3}, we have the following result on the structure of $X_{r+s}$ when $X'_{r+s}\neq 0$, which is crucial to our further discussion.

\begin{cor}\label{X_{r+s}-6} If $X'_{r+s}\ne0$, then there exist $\alpha\in A_n$ and  $i\in\{1,2,\cdots, n\}$ such that
\begin{equation}\label{ca1}
{\rm supp}(X _{r+s})\subset  \{(\beta, i):  \beta\in A_n\},
\end{equation}
or
\begin{equation}\label{ca2}
 {\rm supp}(X _{r+s})\subset   \{(\alpha,k):  1\leq k\leq n\}.
\end{equation}
\end{cor}

\begin{proof}
Let $\Upsilon=\{ 1\le k\le n:\exists \,\alpha\in A_n\text{ such that }(\alpha, k)\in {\rm supp}(X _{r+s})\}$. If $|\Upsilon|=1$, then  (\ref{ca1}) holds. Now suppose that  $|\Upsilon|>1$.

In this case, we may assume $(\alpha, 1), (\beta, 2)\in {\rm supp}(X _{r+s})$ without loss of generality. If $\beta\neq\alpha$, then it follows from
Lemma \ref{X_{r+s}} that
$$X'_{r+s}\in {\rm span}_{\bbf} \{x^\alpha D_k-\alpha_k x^{\alpha+\epsilon_1-\epsilon_k} D_1:k\ne 1\}\cap {\rm span}_{\bbf} \{x^\beta D_l-\beta_l x^{\beta+\epsilon_2-\epsilon_l} D_2:l\ne 2\}.$$
This implies that
$$\beta=\alpha+\epsilon_1-\epsilon_2, \,\,\,(\alpha_1+1)\alpha_2\equiv 1\,(\text{mod}\,p),$$ and there exists some nonzero constant $c\in\bbf$ such that $$X'_{r+s}=c(x^\alpha D_2-\alpha_2 x^{\alpha+\epsilon_1-\epsilon_2} D_1).$$
This implies that $(\alpha+\epsilon_1-\epsilon_2, 1)\in {\rm supp}(X _{r+s})$ and $(\alpha, 2)\in{\rm supp}(X' _{r+s})$ by Lemma \ref{delta lem}. It contradict with Lemma \ref{X_{r+s}-3}. Hence, (\ref{ca2}) holds.
\end{proof}

\begin{cor}\label{X_{r+s}-7} If there exist some $\alpha,\beta\in A_n$ and $i\neq j\in\{1,\cdots, n\}$ such that $$\{(\alpha,i), (\beta, j)\}\subset
 {\rm supp}(X _{r+s}), $$ then $X'_{r+s-1}=0$.
\end{cor}

\begin{proof} The assertion follows directly from Lemma \ref{X_{r+s}}.
\end{proof}

Let $$\mathscr{T}_1:={\rm span}_{\bbf}\{I_1=\sum_{i=1}^n x_iD_i, h_j=x_jD_j+x_1D_j:2\le j\le n\},$$
and when $p>2$,  for $2\leq k\leq n$, set
$$\mathscr{T}_k:={\rm span}_{\bbf}\{I_k=x_kD_k+I_1, h_j=x_jD_j+x_1D_j, \mathbbm{h}_k=x_kD_k+x_1^2D_k :2\le j\le n,\, j\neq k\}.$$

\begin{lem}\label{I_k} We have $ \mathfrak{z}_{\ggg}(\mathscr{T}_k)=\mathscr{T}_k$ for any $k=1,2,\cdots,n$. Moreover,
each $\mathscr{T}_k$ is a Cartan subalgebra  of $\ggg$.
\end{lem}

\begin{proof}
Define the following algebra isomorphisms
\begin{eqnarray*}
\psi_1:\, \mathfrak{A}_n &\longrightarrow &\mathfrak{A}_n\\
x_i&\longmapsto & x_i+(1-\delta_{i1})x_1,\quad 1\leq i\leq n,
\end{eqnarray*}
and for $2\leq k\leq n$,
\begin{eqnarray*}
\psi_k:\, \mathfrak{A}_n &\longrightarrow &\mathfrak{A}_n\\
x_i&\longmapsto & x_i+(1-\delta_{i1})x_1^{1+\delta_{ik}},\quad 1\leq i\leq n.
\end{eqnarray*}
Then it follows from \cite[Theorem 2]{Wi} that these algebra isomorphisms $\psi_k$ ($1\leq k\leq n$) induce the following Lie algebra isomorphisms,
\begin{eqnarray*}
\widetilde{\psi_k}:\, \ggg=\Der(\mathfrak{A}_n) &\longrightarrow & \ggg=\Der(\mathfrak{A}_n)\\
E&\longmapsto &\psi_k\circ E\circ\psi_k^{-1},\,\,\forall\,E\in\ggg.
\end{eqnarray*}
It follows from direct computation that
$$\psi_1(x_1D_1)=x_1\big(D_1-\sum\limits_{j=2}^nD_j\big),\quad \psi_1(x_lD_l)=h_l,\,2\leq l\leq n,$$
and for $2\leq k\leq n$,
$$\psi_k(x_1D_1)=x_1\big(D_1-\sum\limits_{\stackrel{j=2}{j\neq k}}^n D_j\big)-2x_1^2D_k,\,\,\, \psi_k(x_lD_l)=h_l,\, 2\leq l\leq n\,\,\text{and}\,\, l\neq k,\,\,\,\psi_k(x_kD_k)=\mathbbm{h}_k.$$
Therefore, $\psi_k(T)=\mathscr{T}_k$ for any $1\leq k\leq n$, so that these $\mathscr{T}_k$ ($1\leq k\leq n$) are Cartan subalgebras of $\ggg$. Moreover,
$$\mathfrak{z}_{\ggg}(\mathscr{T}_k)=\mathfrak{z}_{\ggg}\big(\psi_k(T)\big)=\psi_k\big(\mathfrak{z}_{\ggg}(T)\big)=\psi_k(T)=\mathscr{T}_k.$$
We complete the proof.
\end{proof}

\begin{lem}\label{X_{r+s}-8} Suppose that
$ {\rm supp}(X _{r+s})\subset\{(\beta, 1): \beta\in A_n\}$. Then
$X'=c[I_1, X]$ for some $c\in\bbf$. In particular, if $X'\neq 0$, then $X_l=0$ for any $l<r+s-1$ with $p\nmid l$; and if  $X'_l\ne0$ then $p \nmid l$. Moreover, if $p>2$ and there exists $(\beta, 1)\in {\rm supp}(X _{r+s})$ with $\beta_1<p-1$, then there exists some $c_k\in\bbf$ such that $X'=c_k[I_k, X],\,2\leq k\leq n.$
\end{lem}

\begin{proof}
For any $1\leq k\leq n$ and any   regular
vector $\lambda\in\bbf^n$, let
$$h_{\lambda}=
\begin{cases}
\lambda_1I_k+\sum\limits_{j=2}^n\lambda_jh_j, &\text{if}\,\,k=1, \vspace{1mm}\cr \lambda_1I_k+\sum\limits_{\stackrel{j=2}{j\neq k}}^n\lambda_jh_j+\lambda_k \mathbbm{h}_k, &\text{if}\,\,2\leq k\leq n\text{ and }p>2.
\end{cases}$$

If $p=2$, for $h_{\lambda}$ and $\mathscr{D}_1$, there exists $u=\sum_{i=1}^{n}a_i\mathscr{D}_i\in\sum_{i=1}^{n}\bbf \mathscr{D}_i$ with $a_i\in \bbf$ such that $\Delta(h_{\lambda})=[u, h_{\lambda}]$ by Lemma
\ref{lem on centralizer}(3). By Lemma \ref{delta lem} we see that
$$0=\bigg[\sum_{i=1}^{n}a_i{D}_i, h_{\lambda}\bigg]=a_1(\lambda_1D_1+\sum_{j=2}^n\lambda_jD_j)+\sum_{j=2}^na_j(\lambda_1+\lambda_j)D_j.
$$
Then it follows $a_i=0$ for all $i=1,2,\cdots, n$, i.e.,    $u=0$, and $\Delta(h_{\lambda})=0$.

 If $p>2$, it follows from Lemma \ref{delta lem} and Corollary \ref{Try2}(2) that $\Delta(h_{\lambda})=0$. For $h_{\lambda}$ and $X$,  there exists $a\in\ggg$ such that
$0=\Delta(h_{\lambda})=[a, h_{\lambda}]$ and $X'=[a, X]$. Then by Lemma \ref{I_k}, there exist $a_i\in\bbf, 1\leq i\leq n$ such that
$$a=
\begin{cases}
a_1I_k+\sum\limits_{j=2}^na_jh_j, &\text{if}\,\,k=1, \vspace{1mm}\cr a_1I_k+\sum\limits_{\stackrel{j=2}{j\neq k}}^na_jh_j+a_k \mathbbm{h}_k, &\text{if}\,\,2\leq k\leq n.
\end{cases}$$
Note that  ${\rm supp}(X'_{r+s})\subseteq {\rm supp}(X _{r+s})$ and $X'_{r+s+1}=0$ by Lemma \ref{delta lem}. If $a_j\ne0$ for any $j\ne1$, then $[a, X]$ has a nonzero term $-a_ja_{\beta, 1}x^{\beta}D_j$ or  $-2a_ja_{\beta, 1}x^{\beta+\epsilon_1}D_j$. It follows that $a_j=0$ for $2\leq j\leq n$, i.e.,
$a=a_1I_k$. Other assertions follow easily.
\end{proof}

\begin{rem}
The result in Lemma \ref{X_{r+s}-8} does not need the assumption that $s$ is minimal for $X$.
\end{rem}

\begin{prop}\label{key prop1}
Suppose ${\rm supp}(X _{r+s})\subset\{(\beta, 1): \beta\in A_n\}$. Then $X'=0$.
\end{prop}

\begin{proof}
Assume $X'\neq 0$, we will deduce some contradictions in the following discussion.

Case 1: $n=1$.

In this case,  we have assumed that $p>2$. It follows from Lemma \ref{X_{r+s}-8} that $X=X_0+X_{r+s-1}+X_{r+s}$. Since $\Delta(D_1)=0$ by Corollary \ref{Try2}(1) and $\mathfrak{z}_{\ggg}(D_1)=\bbf D_1$, there exists some $c\in\bbf^*$ such that
$$X'=\Delta(X)=[cD_1, X].$$ This implies that $X'_{r+s}=0$ and $X'_{r+s-1}\neq 0$, so that $X_{r+s-1}\neq 0$. Again from Lemma \ref{X_{r+s}-8}
we see that $p|r+s$ and $p\nmid r+s-1>0$.
Consequently, $X'_{r+s-2}\neq 0$ which contradicts   Lemma \ref{X'}.

Case 2: $n\geq 2$.

In this case, we first claim that $r+s<n(p-1)-1$. Indeed, if $r+s=n(p-1)-1$, then $X_{r+s}=a_{\tau, 1}x^{\tau}D_1$, where $0\neq a_{\tau, 1}\in\bbf$, $\tau=\sum_{j=1}^n(p-1)\epsilon_j$. It follows from
Lemma \ref{X_{r+s}-8} that $\Delta(D_j-X)=0$ for any $1\leq j\leq n$. Then for $X$ and $D_j-X$, there exists some $a_j\in\ggg$ such that
$0=\Delta(D_j-X)=[a_j, D_j-X]$ and
\begin{equation*}\label{jud1}
X'=\Delta(X)=[a_j, X]=[a_j, D_j],\quad 1\leq j\leq n.
\end{equation*}
The right-hand-side can not produce the term $x^{\tau}D_1$, which
implies that $X'_{r+s}=0$. Hence, $X'_{r+s-1}\neq 0$. From Lemma \ref{X_{r+s}} we know that $X_{r+s-1}$ (also $X'_{r+s-1}$ with different coefficient) has a term $a_{\tau-\epsilon_i, 1}x^{\tau-\epsilon_i}D_1\ne0$ for some $i=1,2,\cdots, n$. Choose $j\in \{1,2,\cdots, n\}\setminus\{i\}$.
It follows from
Lemma \ref{X_{r+s}-8} that $\Delta(D_j-X)=0$. Then for $X$ and $D_j-X$, there exists some $a_j\in\ggg$ such that
$0=\Delta(D_j-X)=[a_j, D_j-X]$ and
\begin{equation}\label{jud1'}
X'=\Delta(X)=[a_j, X]=[a_j, D_j].
\end{equation}
The coefficient of the term $x^{\tau-\epsilon_i}D_1$ on the right hand side of (\ref{jud1'}) is $0$. This implies that $X'_{r+s-1}=0$, a contradiction. Therefore, $r+s<n(p-1)-1$.

According to the discussion above, and the assumption on $X$ at the beginning, we have $\Delta(X-x^{\tau}D_2)=0$. Then for $X$ and $X-x^{\tau}D_2$, there exists
some $a\in\ggg$ such that $$0=\Delta(X-x^{\tau}D_2)=[a, X-x^{\tau}D_2]$$ and $$X'=\Delta(X)=[a, X]=[a, x^{\tau}D_2].$$ This implies that $r+s=n(p-1)-2$, $X'=X'_{r+s}$, and ${\rm supp}(X')\subset\{(\tau-\epsilon_j, 2):1\leq j\leq n\}$. It contradicts with ${\rm supp}(X')\subset {\rm supp}(X)\subset \{(\beta, 1): \beta\in A_n\}$.

In conclusion, we have shown that $X'=0$. The proof is complete.
\end{proof}

\begin{prop}\label{key prop2}
Suppose $ {\rm supp}(X _{r+s})\subset   \{(\alpha,k):  1\le k\le n\}$. Then $X'=0$.
\end{prop}

\begin{proof}
According to the assumption, we can write
$$X_{r+s}=c_1x^{\alpha}D_1+c_2x^{\alpha}D_2+\cdots+c_nx^{\alpha}D_n=x^{\alpha}(c_1D_1+c_2D_2+\cdots+c_nD_n),$$
where $c_i\in\bbf, 1\leq i\leq n$. We can assume that $c_1\neq 0$ without loss of generality. Let ${\mathfrak{B}}_n=\bbf[y_1,\cdots, y_n]/(y_1^p,\cdots, y_n^p)$ be the divided power algebra of $n$ variables $y_1,\cdots, y_n$, where $(y_1^p,\cdots, y_n^p)$ denotes the ideal of $\bbf[y_1,\cdots, y_n]$ generated by $y_i^p, 1\leq i\leq n$. Define the following algebra isomorphism
\begin{eqnarray*}
\varphi:\, \mathfrak{A}_n &\longrightarrow &\mathfrak{B}_n\\
x_i&\longmapsto & c_iy_1+(1-\delta_{i1})y_i,\quad 1\leq i\leq n.
\end{eqnarray*}
Then it induces the following Lie algebra isomorphism
\begin{eqnarray*}
\widetilde{\varphi}:\, \ggg=\Der(\mathfrak{A}_n) &\longrightarrow &\hhh:=\Der(\mathfrak{B}_n)\\
E&\longmapsto &\varphi\circ E\circ\varphi^{-1},\,\,\forall\,E\in\ggg.
\end{eqnarray*}

It follows from a direct computation that for any $\alpha\in A_n$ and $1\leq i\leq n$, we have
$$\widetilde{\varphi}(x^{\alpha}D_i)=
\begin{cases}
\frac{1}{c_1}\prod\limits_{j=1}^n(c_jy_1+(1-\delta_{j1})y_j)^{\alpha_j}\big(\widetilde{D_1}-\sum\limits_{k=2}^n c_k\widetilde{D_k}\big), &\text{if}\,\,i=1, \vspace{1mm} \cr
\prod\limits_{j=1}^n(c_jy_1+(1-\delta_{j1})y_j)^{\alpha_j}\widetilde{D_i}, &\text{if}\,\,2\leq i\leq n,
\end{cases}$$
where $\widetilde{D_i}$ is a derivation on $\mathfrak{B}_n$ defined by $\widetilde{D_i}(y_j)=\delta_{ij}$ for $1\leq i, j\leq n$. The Lie algebra $\hhh$ is a
free $\mathfrak{B}_n$-module of rank $n$ with basis $\widetilde{D_1},\cdots, \widetilde{D_n}$, and it has a natural $\bbz$-grading similar as the Lie algebra $\ggg$. In particular, $Y:=\widetilde{\varphi}(X)=Y_r+\cdots+Y_{r+s}$ with $Y_j\in\hhh_{[j]}$ for $r\leq j\leq r+s$, and $Y_r\neq 0, Y_{r+s}=\widetilde{\varphi}(X_{r+s})=\prod\limits_{j=1}^n(c_jy_1+(1-\delta_{j1})y_j)^{\alpha_j}\widetilde{D_1}\neq 0$.

Moreover, the above Lie algebra homomorphism $\widetilde{\varphi}$ and the 2-local derivation $\Delta$ on $\ggg$ induce a 2-local derivation $\widetilde{\Delta}$ on $\hhh$. Precise speaking,
\begin{eqnarray*}
\widetilde{\Delta}:\, \hhh=\Der(\mathfrak{B}_n) &\longrightarrow &\hhh=\Der(\mathfrak{B}_n)\\
E&\longmapsto &\widetilde{\varphi}\big(\Delta\big(\widetilde{\varphi}^{-1}(E)\big)\big),\,\,\forall\,E\in\hhh.
\end{eqnarray*}
Indeed, for any $E, F\in\hhh$, we have $\widetilde{\varphi}^{-1}(E), \widetilde{\varphi}^{-1}(F)\in\ggg$. Since $\Delta$ is a 2-local derivation on $\ggg$, there exists $D\in\ggg$ such that
$$\Delta\big(\widetilde{\varphi}^{-1}(E)\big)=[D, \widetilde{\varphi}^{-1}(E)],\quad \Delta\big(\widetilde{\varphi}^{-1}(F)\big)=[D, \widetilde{\varphi}^{-1}(F)].$$
Hence,
$$\widetilde{\Delta}(E)=\widetilde{\varphi}\big(\Delta\big(\widetilde{\varphi}^{-1}(E)\big)\big)=[\widetilde{\varphi}(D), E],\quad
\widetilde{\Delta}(F)=\widetilde{\varphi}\big(\Delta\big(\widetilde{\varphi}^{-1}(F)\big)\big)=[\widetilde{\varphi}(D), F].$$
This implies that $\widetilde{\Delta}$ is a 2-local derivation on $\hhh$.

Suppose $X'=\Delta(X)\neq 0$, then $$0\neq\widetilde{\varphi}(X')=\widetilde{\varphi}\big(\Delta\big(\widetilde{\varphi}^{-1}\widetilde{\varphi}(X)\big)\big)=\widetilde{\Delta}
\big(\widetilde{\varphi}(X)\big)=\widetilde{\Delta}(Y).$$
Without loss of generality, we can assume that $Y$ satisfies the same assumption as $X$. Then it follows from Proposition \ref{key prop1} that $\widetilde{\Delta}(Y)=0$, a contradiction. Hence, $X'=0$. We complete the proof.
\end{proof}

\begin{prop}\label{prop for W}
Suppose $\Delta$ is a 2-local derivation on $\ggg$ over a field $\bbf$ of prime characteristic $p$  with cardinality no less than $p^n$ such that $$\Delta\big(\frakd^{(1)}_{\lambda}\big)=\Delta\Big(\sum\limits_{i=1}^nx_i^2D_i\Big)=0\,\,\text{\rm if}\,\, p>2,$$ and
$$\Delta\big(\frakd^{(1)}_{\lambda}\big)=\Delta\big(\mathscr{D}_1\big)=0\,\,\text{\rm if}\,\,p=2$$
for some regular
vector $\lambda\in\bbf^n$. Then $\Delta=0$.
\end{prop}

\begin{proof}
Suppose $X':=\Delta(X)\neq 0$ for some $X=X_r+\cdots+X_{r+s}\in\ggg$ satisfying the assumptions stated in the paragraph before Lemma \ref{X'}, where $X_i\in\ggg_{[i]}$ for $r\leq i\leq r+s$, $ X_{r}\ne0$, $ X_{r+s}\ne0$. Then it follows from Lemma \ref{X'} that $X'= X'_{r+s}+X'_{r+s-1}$ for $X'_j\in\ggg_{[j]}$ for $j=r+s-1, r+s$. We will deduce some contradictions in the following discussion.

Case 1: $X'_{r+s}\neq 0$.

In this case, without loss of generality, there exists some $\alpha\in A_n$ such that one of the following two subcases may happen by Corollary \ref{X_{r+s}-6}.

Subcase 1.1: ${\rm supp}(X _{r+s})\subset  \{(\beta, 1):  \beta\in A_n\}$.

In this subcase, it follows from Proposition \ref{key prop1} that $X'=0$, a contradiction.

Subcase 1.2: $ {\rm supp}(X _{r+s})\subset   \{(\alpha,k):  1\leq k\leq n\}$.

In this subcase, it follows from Proposition \ref{key prop2} that $X'=0$, a contradiction.

Case 2: $X'_{r+s}=0$.

In this case, $X'_{r+s-1}\neq 0$. Thanks to Corollary \ref{X_{r+s}-7}, without loss of generality, we may assume that
$${\rm supp}(X _{r+s})\subset\{(\beta, 1): \beta\in A_n\}.$$
Then it follows from Proposition \ref{key prop1} that $X'=0$, a contradiction.

In conclusion, we have shown that $\Delta=0$. The proof is complete.
\end{proof}

We are now in the position to present the following main result in this section.

\begin{thm}\label{main thm}
Let $\ggg=W_n$ be the simple Jacobson-Witt algebra over a field $\bbf$ of prime characteristic $p$ with cardinality no less than $p^n$. Then every 2-local derivation on $\ggg$ is a derivation.
\end{thm}

\begin{proof}
Let $\Delta$ be a 2-local derivation on $\ggg$. Take a regular vector $\lambda\in\bbf^n$.
Then there exists an element $a\in\ggg$
such that $$\Delta\big(\frakd^{(1)}_{\lambda}\big)=\big[a, \frakd^{(1)}_{\lambda}\big],\quad \Delta\bigg(\sum\limits_{i=1}^n x_i^2D_i\bigg)=\bigg[a, \sum\limits_{i=1}^n x_i^2D_i\bigg],\,\,\text{if}\,\,p>2,$$ and
$$\Delta\big(\frakd^{(1)}_{\lambda}\big)=\big[a, \frakd^{(1)}_{\lambda}\big],\quad \Delta(\mathscr{D}_1)=[a, \mathscr{D}_1],\,\,\text{if}\,\,p=2.$$
Set $\Delta_1=\Delta-\ad a$. Then $\Delta_1$ is a 2-local derivation on $\ggg$ such that  $$\Delta_1\big(\frakd^{(1)}_{\lambda}\big)=\Delta_1\bigg(\sum\limits_{i=1}^nx_i^2D_i\bigg)=0,\,\,\text{if}\,\,p>2,$$ and
$$\Delta_1\big(\frakd^{(1)}_{\lambda}\big)=\Delta_1(\mathscr{D}_1)=0,\,\,\text{if}\,\,p=2.$$
It follows from Proposition \ref{prop for W} that $\Delta_1=0$. Thus $\Delta=\ad a$ is a derivation. The proof is complete.
\end{proof}

The following example implies that the assumption on the simplicity of the Jacobson-Witt algebras in Theorem \ref{main thm} is necessary.
\begin{example}
Let $\ggg=W_1$ be the Witt algebra over a field $\bbf$ of characteristic $p=2$. Then $\ggg$ is a two dimensional solvable Lie algebra with a basis $\{e_{-1}, e_0\}$ and subject to the relation $[e_{-1}, e_0]=e_{-1}$. Note that $W_1$ in this case is not a Lie algebra in Sect.3 of \cite{AKR}. 

For $i=-1,0$, let $\mathbb{D}_{i}=\ad(e_i)$, i.e.,  $\mathbb{D}_i(e_j)=\delta_{ij}e_{-1}$ for $j=-1,0$.   It follows from Lemma \ref{derivation lem} that
$\Der(\ggg)=\bbf \mathbb{D}_{-1}\oplus \bbf \mathbb{D}_{0}$. Let
\begin{eqnarray*}
\Delta:\, \ggg &\longrightarrow & \ggg\\
k_{-1}e_{-1}+k_0e_0&\mapsto &\begin{cases}
k_0e_{-1}, &\text{if}\,\,k_{-1}\neq 0, \vspace{1mm} \cr
0, &\text{if}\,\,k_{-1}=0,
\end{cases}\,\,\forall\,\,k_{-1}, k_0\in\bbf.
\end{eqnarray*}
We will show that $\Delta$ is a 2-local derivation on $\ggg$, but not a derivation. For  any $x=a_{-1}e_{-1}+a_0e_0, y=b_{-1}e_{-1}+b_0e_0\in\ggg$, we consider the following four cases.

Case 1: $a_{-1}=b_{-1}=0$.

In this case, take $D_{xy}=0$. Then
$$\Delta(x)=D_{xy}(x)=0,\quad \Delta(y)=D_{xy}(y)=0.$$

Case 2: $a_{-1}\neq 0, b_{-1}=0$.

In this case, take $D_{xy}=\frac{a_0}{a_{-1}}\mathbb{D}_{-1}\in\Der(\ggg)$. Then
$$\Delta(x)=D_{xy}(x)=a_0e_{-1},\quad \Delta(y)=D_{xy}(y)=0.$$

Case 3: $a_{-1}=0, b_{-1}\neq 0$.

In this case, take $D_{xy}=\frac{b_0}{b_{-1}}\mathbb{D}_{-1}\in\Der(\ggg)$. Then
$$\Delta(x)=D_{xy}(x)=0,\quad \Delta(y)=D_{xy}(y)=b_0e_{-1}.$$

Case 4: $a_{-1}\neq 0$ and $b_{-1}\neq 0$.

In this case, take $D_{xy}=\mathbb{D}_0\in\Der(\ggg)$. Then

$$\Delta(x)=D_{xy}(x)=a_0e_{-1},\quad \Delta(y)=D_{xy}(y)=b_0e_{-1}.$$

Therefore, $\Delta$ is a 2-local derivation on $\ggg$.  However,
$$\Delta(e_{-1}+e_{0})=e_{-1}\neq \Delta(e_{-1})+\Delta(e_0).$$
Hence, $\Delta$ is not a derivation.
\end{example}

\end{document}